\documentclass{amsart}[12pt]
\def\doctype{}

\usepackage{latexsym,amssymb,amsfonts,dsfont,bm}
\usepackage{fancyhdr,caption}
\usepackage{tikz}
\usetikzlibrary{decorations.pathmorphing}

\usepackage{hyperref}

\newcommand\Z{\mathbb{Z}}
\newcommand\F{\mathbb{F}}
\newcommand{\comment}[1]{}

\newcommand{\sinewave}[4][]{\draw[#1]  plot[domain=-1:2] (\x,{#2*sin((#4*pi/180)r + 2*pi*#3*\x r)})}
\newcommand{\vsinewave}[4][]{\draw[#1]  plot[domain=-1:1] ({#2*sin((#4*pi/180)r + 2*pi*#3*\x r)},\x)}

\numberwithin{equation}{section}

\fancypagestyle{titlepage}{
\fancyhead[R]{\doctype}
\fancyhead[CL]{}
\cfoot{\vspace{5pt} \thepage}
}

\let\oldsection\section
\newcommand\boldsection[1]{\oldsection{\bf #1}}
\newcommand\starsection[1]{\oldsection*{\bf #1}}
\makeatletter
\renewcommand\section{\@ifstar\starsection\boldsection}
\makeatother

\newtheoremstyle{theorem}
  {12pt}		  
  {0pt}  
  {\sl}  
  {\parindent}     
  {\bf}  
  {. }    
  { }    
  {}     
\theoremstyle{theorem}
\newtheorem{thm}{Theorem}[section]  
\newtheorem{lemma}[thm]{Lemma}     

\newtheorem{cons}[thm]{Construction}

\newtheoremstyle{definition}
  {12pt}		  
  {0pt}  
  {}  
  {\parindent}     
  {\bf}  
  {. }    
  { }    
  {}     
\theoremstyle{definition}

\newtheorem{ex}[thm]{Example}

\renewcommand{\proofname}{Proof}

\makeatletter
\renewenvironment{proof}[1][\proofname]{\par
  \pushQED{\qed}%
  \normalfont \partopsep=\z@skip \topsep=\z@skip
  \trivlist
  \item[\hskip\labelsep
        \scshape
    #1\@addpunct{.}]\ignorespaces
}{%
  \popQED\endtrivlist\@endpefalse
}
\makeatother

\makeatletter
\renewcommand*\@maketitle{%
  \normalfont\normalsize
  \@adminfootnotes
  \@mkboth{\@nx\shortauthors}{\@nx\shorttitle}%
  \global\topskip42\p@\relax 
  \@settitle
  \ifx\@empty\authors \else {\vskip 1em
\vtop{\centering\shortauthors\@@par}} \fi
  \ifx\@empty\@date \else {\vskip 1em \vtop{\centering\@date\@@par}}\fi 
  \ifx\@empty\@dedicatory
  \else
    \baselineskip18\p@
    \vtop{\centering{\footnotesize\itshape\@dedicatory\@@par}%
      \global\dimen@i\prevdepth}\prevdepth\dimen@i
  \fi
  \@setabstract
  \normalsize
  \if@titlepage
    \newpage
  \else
    \dimen@34\p@ \advance\dimen@-\baselineskip
    \vskip\dimen@\relax
  \fi
} 
\renewcommand*\@adminfootnotes{%
  \let\@makefnmark\relax  \let\@thefnmark\relax
  \ifx\@empty\@subjclass\else \@footnotetext{\@setsubjclass}\fi
  \ifx\@empty\@keywords\else \@footnotetext{\@setkeywords}\fi
  \ifx\@empty\thankses\else \@footnotetext{%
    \def\par{\let\par\@par}\@setthanks}%
  \fi
\thispagestyle{titlepage}
}
\makeatother

\title{\large Number cubes with consecutive line sums}

\author{Peter J.~Dukes and Joanna Niezen}
\address{\rm Mathematics and Statistics, 
University of Victoria, Victoria, BC, Canada}
\email{dukes@uvic.ca, jniezen@uvic.ca}

\thanks{Research of the first author is supported by NSERC grant 312595--2017}

\date{January 9, 2021}

\begin{document}

\begin{abstract}
We settle the existence of certain `anti-magic' cubes using combinatorial block designs and graph decompositions to align a handful of small examples.
\end{abstract}

\maketitle
\hrule


\section{Introduction}

We address the following question:

\medskip
\begin{minipage}{6in}
\textit{Does there exist, for every integer $n \ge 2$, an $n \times n \times n$ cube of nonnegative integers whose $3n^2$ line sums, in the direction of coordinate axes as shown, are the integers $0,1,2,\dots,3n^2-1$?}
\end{minipage}

\begin{figure}[htbp]
\begin{center}
\begin{tikzpicture}
\draw (1,1)--(1,0);
\draw (1,1)--(0,1);
\draw (0,0)--(1,0);
\draw (0,0)--(0,1);

\draw (1.6,1.4)--(1.6,0.4);
\draw (1.6,1.4)--(0.6,1.4);
\draw[dashed] (0.6,0.4)--(1.6,0.4);
\draw[dashed] (0.6,0.4)--(0.6,1.4);

\draw[dashed]  (0,0)--(0.6,0.4);
\draw (1,0)--(1.6,0.4);
\draw (0,1)--(0.6,1.4);
\draw (1,1)--(1.6,1.4);

\draw[fill=white] (0,0) circle [radius=.2];
\draw[fill=white] (0,1) circle [radius=.2];
\draw[fill=white] (1,0) circle [radius=.2];
\draw[fill=white] (1,1) circle [radius=.2];

\draw[fill=white] (0.6,0.4) circle [radius=.2];
\draw[fill=white] (0.6,1.4) circle [radius=.2];
\draw[fill=white] (1.6,0.4) circle [radius=.2];
\draw[fill=white] (1.6,1.4) circle [radius=.2];

\node at (0,0) {$0$};
\node at (0,1) {$0$};
\node at (1,0) {$1$};
\node at (1,1) {$2$};

\node at (0.6,0.4) {$4$};
\node at (0.6,1.4) {$6$};
\node at  (1.6,0.4) {$4$};
\node at (1.6,1.4) {$5$};

\end{tikzpicture}
\hspace{1cm}
\begin{tikzpicture}

\draw(1,-1)--(1,1);
\draw(0,-1)--(0,1);
\draw(-1,-1)--(-1,1);
\draw(1,-1)--(-1,-1);
\draw(1,0)--(-1,0);
\draw(1,1)--(-1,1);

\draw[dashed] (0.6,-0.6)--(0.6,1.4);
\draw (1.6,-0.6)--(1.6,1.4);
\draw[dashed] (-0.4,-0.6)--(-0.4,1.4);
\draw (-0.4,1.4)--(1.6,1.4);
\draw[dashed] (-0.4,0.4)--(1.6,0.4);
\draw[dashed] (-0.4,-0.6)--(1.6,-0.6);

\draw[dashed] (1.2,-0.2)--(1.2,1.8);
\draw[dashed] (0.2,-0.2)--(0.2,1.8);
\draw (2.2,-0.2)--(2.2,1.8);
\draw[dashed] (0.2,-0.2)--(2.2,-0.2);
\draw[dashed] (0.2,0.8)--(2.2,0.8);
\draw (0.2,1.8)--(2.2,1.8);

\draw (-1,1)--(0.2,1.8);
\draw (0,1)--(1.2,1.8);
\draw (1,1)--(2.2,1.8);
\draw[dashed] (-1,0)--(0.2,0.8);
\draw[dashed] (0,0)--(1.2,0.8);
\draw (1,0)--(2.2,0.8);
\draw[dashed] (-1,-1)--(0.2,-0.2);
\draw[dashed] (0,-1)--(1.2,-0.2);
\draw (1,-1)--(2.2,-0.2);

\draw[fill=white] (0,0) circle [radius=.1];
\draw[fill=white] (0,1) circle [radius=.1];
\draw[fill=white] (0,-1) circle [radius=.1];
\draw[fill=white] (1,0) circle [radius=.1];
\draw[fill=white] (1,1) circle [radius=.1];
\draw[fill=white] (1,-1) circle [radius=.1];
\draw[fill=white] (-1,0) circle [radius=.1];
\draw[fill=white] (-1,1) circle [radius=.1];
\draw[fill=white] (-1,-1) circle [radius=.1];

\draw[fill=white] (0.6,0.4) circle [radius=.1];
\draw[fill=white] (0.6,1.4) circle [radius=.1];
\draw[fill=white] (0.6,-0.6) circle [radius=.1];
\draw[fill=white] (1.6,0.4) circle [radius=.1];
\draw[fill=white] (1.6,1.4) circle [radius=.1];
\draw[fill=white] (1.6,-0.6) circle [radius=.1];
\draw[fill=white] (-0.4,0.4) circle [radius=.1];
\draw[fill=white] (-0.4,1.4) circle [radius=.1];
\draw[fill=white] (-0.4,-0.6) circle [radius=.1];

\draw[fill=white] (1.2,0.8) circle [radius=.1];
\draw[fill=white] (1.2,1.8) circle [radius=.1];
\draw[fill=white] (1.2,-0.2) circle [radius=.1];
\draw[fill=white] (2.2,0.8) circle [radius=.1];
\draw[fill=white] (2.2,1.8) circle [radius=.1];
\draw[fill=white] (2.2,-0.2) circle [radius=.1];
\draw[fill=white] (0.2,0.8) circle [radius=.1];
\draw[fill=white] (0.2,1.8) circle [radius=.1];
\draw[fill=white] (0.2,-0.2) circle [radius=.1];
\end{tikzpicture}
\end{center}
\end{figure}
After a few moments of thought, the reader might find a solution for $n=2$ similar to the one shown.  The next size, $n=3$, can be settled with a little more persistence, or perhaps with the help of a computer program.  A solution for this (and other small examples) are left for the appendix so as not to spoil the fun of the puzzle.  Our main result, stated later as Theorem~\ref{main}, answers the above question in the affirmative.

The two-dimensional version of the problem is fairly easy to settle. Consider integers arranged in an $n \times n$ grid.  The sum of all of the row sums is equal to the sum of all of the column sums.  So these line sums can exhaust the consecutive integers $0,1,2,\dots,2n-1$ only if their total, $\frac{1}{2}(2n-1)(2n)=n(2n-1)$, is even. That is, $n$ must be even to admit a number `square' with consecutive line sums.  And indeed, for $n$ even, the array shown below gives a solution.  (Blank entries are zero.)
$$
\begin{array}{|ccccccc|}
\hline
0 & 0 &&&&& \\
1 & 2 &&&&&\\
&& 2 & 2 &&&\\
&& 3 & 4 &&&\\
&&&& \ddots &&\\
&&&&& \mbox{\small $n-2$}& \mbox{\small $n-2$} \\
&&&&& \mbox{\small $n-1$} & \mbox{\small $n$} \\
\hline
\end{array}
$$

Returning to our (three-dimensional) problem, such number cubes have relationships with combinatorial designs and graph decompositions.  We tour several topics in these areas while setting up our constructions.  A size-$n$ number cube with line sums $0,1,2,\dots,3n^2-1$ is called a \emph{Sarvate--Beam cube},  abbreviated SBC$(n)$, for its connection with a combinatorial design variant introduced by D.G.~Sarvate and W.~Beam,  \cite{SB1}.

\section{Some related objects}
\subsection{Latin squares}

A \emph{Latin square of order} $n$ is a $n \times n$ array with entries from an $n$-element set (often assumed to be $[n]:=\{1,2,\dots,n\}$), such that every element appears exactly once in each row and each column.  A latin square of order $n=3$ is given below, and it is not hard to extend the circulant construction to any positive integer $n$.
$$\begin{array}{|ccc|}
\hline
1 & 2 & 3 \\
2 & 3 & 1 \\
3 & 1 & 2 \\
\hline
\end{array}$$

Given a Latin square $L$ of order $n$, we may define a $\{0,1\}$-valued cube $\widehat{L}: [n]^3 \rightarrow \{0,1\}$ such that 
$$\widehat{L}(i,j,k)=
\begin{cases}
1 & \text{if } L_{ij}=k, \\
0 & \text{otherwise}.
\end{cases}$$
It is simple to check that each of the $3n^2$ line sums of $\widehat{L}$ equals $1$.  As a result, we may add a multiple of $\widehat{L}$ (entrywise) to an SBC$(n)$ to produce a cube with consecutive line sums starting at any nonnegative integer.  We denote an $n \times n \times n$ cube with line sums $a,a+1,a+2,\dots,a+3n^2-1$ by SBC$_a(n)$.

Latin squares also facilitate a construction to ``inflate'' the size of our number cubes.

\begin{lemma}[see also \cite{DSG}]
\label{product}
If there exists an SBC$(n)$, then there exists an SBC$(mn)$ for every positive integer $m$.
\end{lemma}

\begin{proof}
Let $L$ be a Latin square of order $m$ with entries in $\{0,1,\dots,m-1\}$ and, from the remarks above, let $C_t$ be an SBC$_{3tn^2}(n)$ for $t=0,1,2,\dots,m^2-1$. Arrange these $m^2$ cubes in an $m \times m \times m$ array by putting $C_{km+i}$ in position $(i,j,k)$ whenever $L_{ij} = k$.  With all other entries equal to zero, the line sums in the size-$mn$ cube are precisely the line sums in the size-$m$ cubes, which altogether cover the interval from $0$ to $3m^2n^2-1$.
\end{proof}

We illustrate the method by building an SBC$(6)$.  First, we take nine copies of an SBC$_{12t}(2)$ with abutting line sums, as shown at left for $t=0,1,\dots,8$.  These are placed (in any order) at the preimages $(i,j,k) \in \widehat{L}^{-1}(1)$ arising from a latin square $L$ of order 3, shown with solid dots at right.  The rest of the $6 \times 6 \times 6$ cube is filled with zeros.

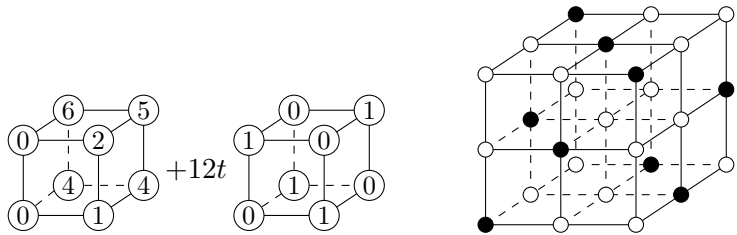
\begin{figure}[htbp]
\begin{center}
\begin{tikzpicture}
\draw (1,1)--(1,0);
\draw (1,1)--(0,1);
\draw (0,0)--(1,0);
\draw (0,0)--(0,1);

\draw (1.6,1.4)--(1.6,0.4);
\draw (1.6,1.4)--(0.6,1.4);
\draw[dashed] (0.6,0.4)--(1.6,0.4);
\draw[dashed] (0.6,0.4)--(0.6,1.4);

\draw[dashed]  (0,0)--(0.5,0.4);
\draw (1,0)--(1.6,0.4);
\draw (0,1)--(0.6,1.4);
\draw (1,1)--(1.6,1.4);

\draw[fill=white] (0,0) circle [radius=.2];
\draw[fill=white] (0,1) circle [radius=.2];
\draw[fill=white] (1,0) circle [radius=.2];
\draw[fill=white] (1,1) circle [radius=.2];

\draw[fill=white] (0.6,0.4) circle [radius=.2];
\draw[fill=white] (0.6,1.4) circle [radius=.2];
\draw[fill=white] (1.6,0.4) circle [radius=.2];
\draw[fill=white] (1.6,1.4) circle [radius=.2];

\node at (0,0) {$0$};
\node at (0,1) {$0$};
\node at (1,0) {$1$};
\node at (1,1) {$2$};

\node at (0.6,0.4) {$4$};
\node at (0.6,1.4) {$6$};
\node at  (1.6,0.4) {$4$};
\node at (1.6,1.4) {$5$};

\node at (2.3,0.6){\large $+12t$};

\draw (4,1)--(4,0);
\draw (4,1)--(3,1);
\draw (3,0)--(4,0);
\draw (3,0)--(3,1);

\draw (4.6,1.4)--(4.6,0.4);
\draw (4.6,1.4)--(3.6,1.4);
\draw[dashed] (3.6,0.4)--(4.6,0.4);
\draw[dashed] (3.6,0.4)--(3.6,1.4);

\draw[dashed]  (3,0)--(3.6,0.4);
\draw (4,0)--(4.6,0.4);
\draw (3,1)--(3.6,1.4);
\draw (4,1)--(4.6,1.4);

\draw[fill=white] (3,0) circle [radius=.2];
\draw[fill=white] (3,1) circle [radius=.2];
\draw[fill=white] (4,0) circle [radius=.2];
\draw[fill=white] (4,1) circle [radius=.2];

\draw[fill=white] (3.6,0.4) circle [radius=.2];
\draw[fill=white] (3.6,1.4) circle [radius=.2];
\draw[fill=white] (4.6,0.4) circle [radius=.2];
\draw[fill=white] (4.6,1.4) circle [radius=.2];

\node at (3,0) {$0$};
\node at (3,1) {$1$};
\node at (4,0) {$1$};
\node at (4,1) {$0$};

\node at (3.6,0.4) {$1$};
\node at (3.6,1.4) {$0$};
\node at (4.6,0.4) {$0$};
\node at (4.6,1.4) {$1$};
\end{tikzpicture}
\hspace{1cm}
\begin{tikzpicture}

\draw(1,-1)--(1,1);
\draw(0,-1)--(0,1);
\draw(-1,-1)--(-1,1);
\draw(1,-1)--(-1,-1);
\draw(1,0)--(-1,0);
\draw(1,1)--(-1,1);

\draw[dashed] (0.6,-0.6)--(0.6,1.4);
\draw (1.6,-0.6)--(1.6,1.4);
\draw[dashed] (-0.4,-0.6)--(-0.4,1.4);
\draw (-0.4,1.4)--(1.6,1.4);
\draw[dashed] (-0.4,0.4)--(1.6,0.4);
\draw[dashed] (-0.4,-0.6)--(1.6,-0.6);

\draw[dashed] (1.2,-0.2)--(1.2,1.8);
\draw[dashed] (0.2,-0.2)--(0.2,1.8);
\draw (2.2,-0.2)--(2.2,1.8);
\draw[dashed] (0.2,-0.2)--(2.2,-0.2);
\draw[dashed] (0.2,0.8)--(2.2,0.8);
\draw (0.2,1.8)--(2.2,1.8);

\draw (-1,1)--(0.2,1.8);
\draw (0,1)--(1.2,1.8);
\draw (1,1)--(2.2,1.8);
\draw[dashed] (-1,0)--(0.2,0.8);
\draw[dashed] (0,0)--(1.2,0.8);
\draw (1,0)--(2.2,0.8);
\draw[dashed] (-1,-1)--(0.2,-0.2);
\draw[dashed] (0,-1)--(1.2,-0.2);
\draw (1,-1)--(2.2,-0.2);

\draw[fill=black] (0,0) circle [radius=.1];
\draw[fill=white] (0,1) circle [radius=.1];
\draw[fill=white] (0,-1) circle [radius=.1];
\draw[fill=white] (1,0) circle [radius=.1];
\draw[fill=black] (1,1) circle [radius=.1];
\draw[fill=white] (1,-1) circle [radius=.1];
\draw[fill=white] (-1,0) circle [radius=.1];
\draw[fill=white] (-1,1) circle [radius=.1];
\draw[fill=black] (-1,-1) circle [radius=.1];

\draw[fill=white] (0.6,0.4) circle [radius=.1];
\draw[fill=black] (0.6,1.4) circle [radius=.1];
\draw[fill=white] (0.6,-0.6) circle [radius=.1];
\draw[fill=white] (1.6,0.4) circle [radius=.1];
\draw[fill=white] (1.6,1.4) circle [radius=.1];
\draw[fill=black] (1.6,-0.6) circle [radius=.1];
\draw[fill=black] (-0.4,0.4) circle [radius=.1];
\draw[fill=white] (-0.4,1.4) circle [radius=.1];
\draw[fill=white] (-0.4,-0.6) circle [radius=.1];

\draw[fill=white] (1.2,0.8) circle [radius=.1];
\draw[fill=white] (1.2,1.8) circle [radius=.1];
\draw[fill=black] (1.2,-0.2) circle [radius=.1];
\draw[fill=black] (2.2,0.8) circle [radius=.1];
\draw[fill=white] (2.2,1.8) circle [radius=.1];
\draw[fill=white] (2.2,-0.2) circle [radius=.1];
\draw[fill=white] (0.2,0.8) circle [radius=.1];
\draw[fill=black] (0.2,1.8) circle [radius=.1];
\draw[fill=white] (0.2,-0.2) circle [radius=.1];
\end{tikzpicture}
\end{center}
\end{figure}

Although Lemma~\ref{product} reduces our problem of finding SBC$(n)$ to prime values $n$, a direct construction for primes has eluded our efforts.  Later, we introduce a more powerful construction that works equally well for both prime and composite integers.

\subsection{Designs}

Let $v$ be a positive integer and $K \subseteq \{2,3,4,\dots\}$.
A \emph{pairwise balanced design} PBD$(v,K)$ is a pair $(V,\mathcal{B})$, where 
\begin{itemize}
\item
$V$ is a set with $|V|=v$;
\item
$\mathcal{B}$ is a family of subsets of $V$, called \emph{blocks}, where $|B| \in K$ for every $B \in \mathcal{B}$; and
\item
any two distinct elements of $V$ appear together in exactly one block.
\end{itemize}
A special class of designs useful for our constructions to follow are the finite planes.  From a finite field $\F_q$ of order $q$, we may construct an  \emph{affine plane}, with elements $\F_q^2$ and blocks given by the affine lines.  This produces a pairwise balanced design PBD$(q^2,\{q\})$.  If each of the $q+1$ parallel classes of lines is projectively extended, the result is a \emph{projective plane} of order $q$, which is also a PBD$(q^2+q+1,\{q+1\})$.

A PBD$(v,\{3\})$ is also known as a \emph{Steiner triple system}; these are known to exist for all positive integers $v \equiv 1,3 \pmod{6}$.  Two interesting small examples come from finite planes.  A Steiner triple system with $v=7$ arises from the projective plane of order $q=2$.  An explicit construction on $V=\Z/7\Z$ comes from developing the ``base block'' $\{0,1,3\}$ additively mod $7$, producing a family $\mathcal{B}$ of seven $3$-subsets covering every pair exactly once.  A Steiner triple system with $v=9$ arises from the affine plane of order $q=3$.  Here, a convenient presentation is to take as points the nine elements of a $3 \times 3$ grid, and as blocks the rows, columns, diagonals, and broken diagonals, giving twelve $3$-subsets in $\mathcal{B}$.
More information on finite planes and Steiner triple systems, including their historical origins, can be found in \cite{handbook}.

In general, Wilson's theory \cite{Wilson} says that pairwise balanced designs PBD$(v,K)$ exist for all $v$ greater than some constant $v_0(K)$, provided certain congruence conditions hold.  The congruence conditions disappear if, for instance, $K$ contains three consecutive integers. To illustrate Wilson's theory in this case, and for later reference, we cite the following result.
\begin{lemma}[\cite{pbd456}]
\label{pbd456}
There exists a PBD$(v,\{4,5,6\})$ for all $v \ge 24$, and also for $v \in \{4,5,6,13,16,17,$ $20,21,22\}$.
\end{lemma}

In \cite{SB1}, a curious variant of Steiner triple systems was introduced.
A \emph{Sarvate--Beam triple system} of order $v$, or SBTS$(v)$, is an assignment $f:\binom{[v]}{3} \rightarrow \Z_{\ge 0}$ of nonnegative integer weights to the $3$-subsets of $[v]$ such that
\begin{equation}
\label{sbts-cond}
\{\tilde{f}(\{i,j\}) : 1 \le i < j \le v\} = \{0,1,2,\dots,\tbinom{v}{2}-1\},
\end{equation}
where $\tilde{f}(\{i,j\}) = \sum_{k \neq i,j} f(\{i,j,k\})$.
That is, an SBTS$(v)$ is a multiset of $3$-element subsets with the property that the induced frequencies on $2$-element subsets cover the range of values from $0$ to $\binom{v}{2}-1$, inclusive.  Note that if the right side of \eqref{sbts-cond} is changed to $\{1\}$, we recover the definition of a Steiner triple system.  For example, an SBTS$(5)$ arises from the following assignment of (positive) block multiplicities.

\begin{center}
\begin{tabular}{ccccccc}
\hline
block & $\{2,3,4\}$ &
$\{1,3,5\}$ &
$\{2,3,5\}$ &
$\{1,4,5\}$ &
$\{2,4,5\}$ &
$\{3,4,5\}$ \\
\hline
multiplicity & 2 & 1 & 3 & 2 & 5 & 2\\
\hline
\end{tabular}
\end{center}
 
See also \cite{St2} for a similar example and \cite{DSG} for a construction for each $v \ge 4$.

\subsection{Graph and multigraph decompositions}

Taking an alternative viewpoint, a Steiner triple system on $v$ elements is equivalent to an edge-decomposition of the complete graph $K_v$ into triangles.  Correspondingly, an SBTS$(v)$ is an edge decomposition into triangles of a certain multigraph on $v$ vertices, namely one whose edge multiplicities are $0,1,2,\dots,\binom{v}{2}-1$ (in some arrangement).

For a simple graph $G$ and nonnegative integer $a$, let $\Gamma_a(G)$ denote the set of all multigraphs obtained by assigning multiplicities $a,a+1,\dots,a+|E(G)|-1$ to the edges of $G$.  Let $\Delta_a(G)$ denote the subset of $\Gamma_a(G)$ consisting of those graphs that admit a triangle decomposition.  An SBTS$(v)$ exists if and only if $\Delta_0(K_v) \neq \emptyset$.

Consider now the complete $3$-partite graph $K_{n,n,n}$.  A triangle decomposition of this graph is equivalent to a Latin square of order $n$: simply associate each triangle, say on vertices $\{i_1,j_2,k_3\}$ (where subscripts indicate partite sets), with the placement of symbol $k$ in position $(i,j)$.
Similar to the above, the existence of a Sarvate--Beam cube of order $n$ is equivalent to showing $\Delta_0(K_{n,n,n}) \neq \emptyset$.

The direct product graph $K_3 \times K_n$ resembles $K_{n,n,n}$, except that the former is missing the edges from a family of $n$ vertex-disjoint triangles.  A Latin square $L$ whose diagonal entries satisfy $L_{ii}=i$ is called \emph{idempotent}.  It is easy to see that an idempotent Latin square of order $n$ exists for all positive integers $n \neq 2$.  Using the off-diagonal entries of an idempotent Latin square of order $n$, we see that the graph $K_3 \times K_n$ has a triangle decomposition for all $n \neq 2$.

For our construction to follow, it is helpful to have an SBC-like variant based on $K_3 \times K_n$.  We define a \emph{holey Sarvate--Beam cube}, or SBHC$(n,1^n)$ to be a size-$n$ cube where entry $(i,j,k)$ is filled with a nonnegative integer if and only if $i,j,k$ are pairwise distinct, and such that the $3n(n-1)$ nonempty line sums are $0,1,2,\dots,3n(n-1)-1$.  An SBHC$(n,1^n)$ is equivalent to a triangle decomposition of a multigraph in the class $\Delta_0(K_3 \times K_n)$.  The appendix gives an SBHC$(n,1^n)$ for each $n \in \{4,5,6\}$.

Let $J_n$ denote the graph $K_3 \times K_n$ with one additional triangle, as depicted in Figure~\ref{sbjc} for $n \in \{3,4\}$.

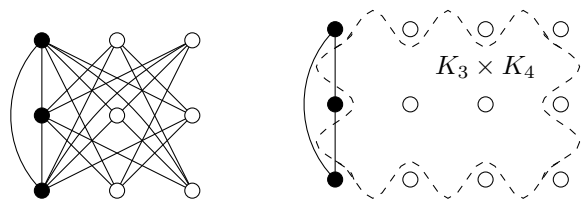
\begin{figure}[htbp]
\begin{center}
\begin{tikzpicture}
\draw (-1,-1)--(-1,1);
\draw (1,1)--(-1,-1);
\draw (1,-1)--(-1,1);
\draw (-1,0)--(1,1);
\draw (-1,0)--(1,-1);
\draw (-1,-1)--(1,0);
\draw (-1,1)--(1,0);

\draw (-1,-1)--(0,1);
\draw (0,-1)--(1,1);
\draw (-1,1)--(0,-1);
\draw (0,1)--(1,-1);

\draw (-1,-1) to [out=60,in=210] (1,1);
\draw (-1,1) to [out=330,in=120] (1,-1);

\draw (0,1)--(-1,0);
\draw (-1,-1) to [out=135,in=225] (-1,1);

\draw (0,-1)--(1,0);
\draw (0,1)--(1,0);
\draw (0,-1)--(-1,0);

\draw[fill=white] (0,0) circle [radius=.1];
\draw[fill=white] (0,1) circle [radius=.1];
\draw[fill=white] (0,-1) circle [radius=.1];
\draw[fill=white] (1,0) circle [radius=.1];
\draw[fill=white] (1,1) circle [radius=.1];
\draw[fill=white] (1,-1) circle [radius=.1];
\filldraw (-1,0) circle [radius=.1];
\filldraw (-1,1) circle [radius=.1];
\filldraw (-1,-1) circle [radius=.1];
\end{tikzpicture}
\hspace{1cm}
\begin{tikzpicture}
\draw (-1,-1)--(-1,1);
\draw (-1,-1) to [out=135,in=225] (-1,1);

\filldraw (-1,0) circle [radius=.1];
\filldraw (-1,1) circle [radius=.1];
\filldraw (-1,-1) circle [radius=.1];
\draw[fill=white] (0,0) circle [radius=.1];
\draw[fill=white] (0,1) circle [radius=.1];
\draw[fill=white] (0,-1) circle [radius=.1];
\draw[fill=white] (1,0) circle [radius=.1];
\draw[fill=white] (1,1) circle [radius=.1];
\draw[fill=white] (1,-1) circle [radius=.1];
\draw[fill=white] (2,0) circle [radius=.1];
\draw[fill=white] (2,1) circle [radius=.1];
\draw[fill=white] (2,-1) circle [radius=.1];

\sinewave[dashed,yshift=1cm]{0.25}{1}{-90};
\sinewave[dashed,yshift=-1cm]{0.25}{1}{90};
\vsinewave[dashed,xshift=-1cm]{0.25}{1}{90};
\vsinewave[dashed,xshift=2cm]{0.25}{1}{-90};

\node at (1,0.5) {$K_3 \times K_4$};

\end{tikzpicture}
\caption{The graphs $J_3$ and $J_4$.}
\label{sbjc}
\end{center}
\end{figure}
Following the notation for holey Sarvate--Beam cubes above, it is natural to use the notation SBHC$(n,1^{n-1})$ for a triangle decomposition (in cube form) of a multigraph in $\Delta_0(J_n)$, $n=3,4$.  An example cube for each of these sizes appears in the appendix.

\section{A construction}

We illustrate the main idea of our construction to follow by building a solution to our problem for $n=7$.  In a little more detail, we 
use a template Steiner triple system on $7$ elements and a graph in $\Delta_0(J_3)$ to construct an SBC$(7)$.  

\begin{ex} 
\label{7} 
We claim that $K_{7,7,7}$ decomposes into seven copies of $J_3$, using a Steiner triple system 
with blocks
$$
\{\underline{0},1,3\},
\{\underline{1},2,4\},
\{\underline{2},3,5\},
\{\underline{3},4,6\},
\{\underline{4},5,0\},
\{\underline{5},6,1\},
\{\underline{6},0,2\}.$$
Each element is replaced by three vertices; each block is to be replaced with a copy of $J_3$ so that the underlined element takes the role of the three filled vertices of $J_3$ in Figure~\ref{sbjc}.
Consider an edge $(i_1,j_2)$ of $K_{7,7,7}$.  If $i=j$, this edge appears in the copy of $J_3$ on the block in which $i$ is underlined.  Otherwise, if $i \neq j$, 
this edge is in the copy of $J_3$ corresponding to the block of the triple system containing $\{i,j\}$.

Finally, we turn $K_{7,7,7}$ into a triangle-decomposable multigraph as follows: the first copy of $J_3$ is replaced by a graph in $\Delta_0(J_3)$, the second by a graph in $\Delta_{21}(J_3)$, the third by a graph in $\Delta_{42}(J_3)$, and so on.  Note that each graph in $\Delta_a(J_k)$ can be obtained from a corresponding graph in $\Delta_0(J_k)$ by increasing the multiplicity of each edge by $a$.  Because every edge of $K_{7,7,7}$ occurs in some copy of $J_3$, and because starting values were chosen as multiples of $|E(J_3)|=21$, we have constructed a graph in $\Delta_{0}(K_{7,7,7})$.
\end{ex}

We present a general construction that captures the above technique.
 
\begin{cons}
\label{cons}
Suppose $\{G_1,\dots,G_b\}$ is an edge decomposition of $G$.  Suppose, for each $i=1,\dots,b$, the graph $G_i$ has a triangle decomposition and also that
$\Delta_0(G_i) \neq \emptyset$.
Then $\Delta_0(G) \neq \emptyset$; that is, some multigraph in $\Gamma_0(G)$ has a triangle decomposition.
\end{cons}

\begin{proof}
Put $m_i=|E(G_i)|$ for each $i$, and 
$a_i=\sum_{j<i} m_j$.  
Since each $G_i$ is simple and has a triangle decomposition, we may increase all edge multiplicities of a graph in $\Delta_0(G_i)$
to produce a graph in $\Delta_a(G_i)$ for any nonnegative integer $a$.  Let $H_i$ denote a multigraph in $\Delta_{a_i}(G_i)$.

We have that $\{H_1,\dots,H_b\}$ is an edge decomposition of some graph in $\Gamma_0(G)$, call it $H$, since the edges of $G$ occur with multiplicities $0,\dots,a_1-1$ in $H_1$, $a_1,\dots,a_2-1$ in $H_2$, and so on, until $a_{b-1},\dots,a_b-1$ in $H_b$.  Since each $H_i$ has a triangle decomposition, it follows that $H$ does as well.
\end{proof}

\section{Solution of the problem}

We are now ready for our main result.
The proof essentially consists of a few direct constructions and a prescription to apply Construction~\ref{cons} for $G=K_{n,n,n}$.  The necessary `building block' constructions can be found in the appendix.

\begin{thm}
\label{main}
There exists an SBC$(n)$ for every integer $n \ge 2$.
\end{thm}

\begin{proof}
Suppose first that there exists a PBD$(n+1,\{4,5,6\})$.  We claim there exists an SBC$(n)$.

If we delete one element from the hypothesized PBD, the result is a pairwise balanced design on $n$ elements with a `parallel class' $\mathcal{A}$ of blocks with sizes in $\{3,4,5\}$, and other block sizes in $\{4,5,6\}$.  Replace every element $x$ with three vertices $x_1,x_2,x_3$.
Replace every block $B$, where $|B|=m \in \{3,4,5,6\}$, either by the graph $K_{m,m,m}$ if $B \in \mathcal{A}$, or by $K_3 \times K_m$ otherwise.
The result is an edge decomposition of $K_{n,n,n}$ into graphs isomorphic to $K_{m,m,m}$ for $m \in \{3,4,5\}$, and $K_3 \times K_m$ for $m \in \{4,5,6\}$.  From Construction~\ref{cons} and our examples given in the appendix, it follows that $\Delta_0(K_{n,n,n}) \neq \emptyset$; that is, there exists an SBC$(n)$.

In view of Lemma~\ref{pbd456}, the preceding construction leaves as exceptions $n \in \{7,\dots,11,13,14,17,18,22\}$.  Except for $n=7,11,13,17$, each of these values has a prime divisor less than or equal to $5$, and so an SBC$(n)$ exists by Lemma~\ref{product} and the Sarvate--Beam cubes of order $2,3,$ and $5$ given in the appendix.  An SBC$(7)$ was built in Example \ref{7}.  We turn to the remaining three values.

For $n=11$, we start with an affine plane of order three, or PBD$(9,\{3\})$, and extend two parallel classes using one new element for each to produce a PBD$(11,\{2,3,4\})$. The block set is, say,
\begin{center}
\begin{tabular}{lllll}
$\{\underline{1},2,3,\infty_1\},$ & $\{1,\underline{4},7,\infty_2\},$ & $\{1,\underline{5},9\},$ & $\{1,6,8\}$, \\
$\{4,5,\underline{6},\infty_1\},$ & $\{\underline{2},5,8,\infty_2\},$ &  $\{2,6,\underline{7}\},$ & $\{2,4,9\}$, \\
$\{7,\underline{8},9,\infty_1\},$ & $\{3,6,\underline{9},\infty_2\},$ & $\{\underline{3},4,8\},$ & $ \{3,5,7\}$, & $\{\underline{\infty_1},\underline{\infty_2}\}$.
\end{tabular}
\end{center}

Blocks with one underlined element are replaced with $J_3$ or $J_4$ according to the block size. The underlined element takes the role of filled vertices as in Example~\ref{7}. The remaining blocks of size three are replaced with $K_3 \times K_3$. The block of size two is replaced with $K_{2,2,2}$.  The above set of graphs decompose $K_{11,11,11}$, and each satisfies the assumptions of Construction~\ref{cons}.  This gives an SBC$(11)$.

For $n=13$, we take a projective plane of order three, or PBD$(13,\{4\})$, with blocks
$$x+\{\underline{0},1,3,9\},~x=0,1,2,\dots,12,$$
where addition, mod 13, distributes into the block.  Using this cyclic structure, it is possible to underline exactly one element from each block, as shown.  Accordingly, replace each block with a copy of $J_4$.  This gives a covering of the edges of $K_{13,13,13}$ as needed for 
Construction~\ref{cons}, and hence an SBC$(13)$.

For $n=17$, we start with an affine plane of order four and extend one parallel class to produce a PBD$(17,\{4,5\})$. Explicitly, the blocks can be taken as
\begin{center}
\small
\begin{tabular}{lllll}
$\{ \underline{1}, \underline{2}, \underline{3}, \underline{4}, \underline{\infty}\},$& $\{1, \underline{5}, 9, 16\},$&$\{2, 5, 11, \underline{15}\},$& $\{3, 5, 7, \underline{12}\},$&$\{ 4, 5, 8,14\},$\\
$\{ 5, 6, 10, 13, \infty\},$& $\{1,\underline{ 6}, 7, 8\},$&$\{ 2, 6, 12, \underline{14}\},$& $\{3, 6, \underline{9}, 15\},$&$\{ 4, 6, 11, 16\},$\\
$\{7, 9, 11, 14, \infty\},$& $\{1, \underline{10}, 14, 15\},$&$\{ 2, \underline{7},10,  16\},$ &$\{ 3, 8, 10, \underline{11}\},$&$\{ 4, 9, 10, 12\},$\\
$\{ 8, 12, 15, 16,\infty\},$&$\{ 1, 11, 12,\underline{13}\},$&$\{ 2, \underline{8},9,13\},$&$\{3, 13, 14, \underline{16}\},$&$\{4, 7, 13, 15\}$.
\end{tabular}
\end{center}
We have again underlined each element exactly once.  The first block is to be replaced by $K_{5,5,5}$.  Blocks with exactly one underlined element are to be replaced by $J_4$.  Other blocks are replaced by $K_3 \times K_4$ and $K_3 \times K_5$ according to the size.  Similar to before, we have covered the edges of $K_{17,17,17}$ so as to satisfy the hypotheses of Construction~\ref{cons}.  The result is an SBC$(17)$.
\end{proof}

\section*{Appendix}

Apart from the simple case of SBC$(2)$, our proof requires only seven explicit cubes or cube-variants.  
Layers of the $n\times n \times n$ cube are given in separate $n \times n$ grids and line sums appear in bold text.

\medskip
\noindent
SBC$(3)$ \hspace{3mm} (underlying graph $K_{3,3,3}$; see also \cite{DSG})

\medskip
\noindent
\tabcolsep=3pt
\begin{tabular}{ccc|c}
   0 &    1 &   12 & \textbf{13} \\
   2 &    3 &    1 & \textbf{6} \\
   0 &   10 &    5 & \textbf{15} \\
\hline
\textbf{2}  & \textbf{14}  & \textbf{18}
\end{tabular}
\hspace{5mm}
\begin{tabular}{ccc|c}
   0 &   14 &    8 & \textbf{22} \\
   0 &    5 &    0 & \textbf{5} \\
   3 &    7 &    0 & \textbf{10} \\
\hline
\textbf{3}  & \textbf{26}  & \textbf{8}
\end{tabular}
\hspace{5mm}
\begin{tabular}{ccc|c}
   0 &    1 &    0 & \textbf{1} \\
   7 &   11 &    3 & \textbf{21} \\
   4 &    0 &   20 & \textbf{24} \\
\hline
\textbf{11}  & \textbf{12}  & \textbf{23}
\end{tabular}
\hspace{5mm}
\begin{tabular}{ccc}
\textbf{0} & \textbf{16} & \textbf{20}\\
\textbf{9} & \textbf{19} & \textbf{4}\\
\textbf{7} & \textbf{17} & \textbf{25}\\
~
\end{tabular}
\normalsize

\noindent
SBC$(5)$ \hspace{3mm} (underlying graph $K_{5,5,5}$)

\medskip
\noindent
\small
\tabcolsep=3pt
\begin{tabular}{ccccc|c}
  16 &    9 &    6 &    9 &    1 & \textbf{41} \\
  14 &    0 &    0 &    6 &   27 & \textbf{47} \\
   4 &   13 &    5 &   14 &    0 & \textbf{36} \\
  20 &    2 &    0 &    2 &    5 & \textbf{29} \\
   7 &    8 &   41 &    0 &    4 & \textbf{60} \\
\hline
\textbf{61}  & \textbf{32}  & \textbf{52}  & \textbf{31}  & \textbf{37}
\end{tabular}
\hspace{5mm}
\begin{tabular}{ccccc|c}
   0 &    0 &   10 &    2 &    0 & \textbf{12} \\
   1 &    2 &    2 &    1 &    0 & \textbf{6} \\
   0 &    0 &    0 &    2 &    0 & \textbf{2} \\
   0 &   60 &   14 &    0 &    0 & \textbf{74} \\
   0 &    7 &    8 &   49 &    3 & \textbf{67} \\
\hline
\textbf{1}  & \textbf{69}  & \textbf{34}  & \textbf{54}  & \textbf{3}
\end{tabular}
\hspace{5mm}
\begin{tabular}{ccccc|c}
  15 &   34 &    2 &    6 &   15 & \textbf{72} \\
  35 &    0 &    3 &    0 &    2 & \textbf{40} \\
  13 &   13 &   16 &    1 &    2 & \textbf{45} \\
   0 &    0 &    0 &    0 &    0 & \textbf{0} \\
   8 &    2 &    7 &    0 &    1 & \textbf{18} \\
\hline
\textbf{71}  & \textbf{49}  & \textbf{28}  & \textbf{7}  & \textbf{20}
\end{tabular}

\medskip
\noindent
\begin{tabular}{ccccc|c}
  10 &    0 &    0 &    0 &    5 & \textbf{15} \\
   1 &    6 &    5 &    5 &    0 & \textbf{17} \\
   2 &    0 &    0 &   39 &   14 & \textbf{55} \\
   5 &    2 &   32 &   12 &    0 & \textbf{51} \\
  41 &    5 &    1 &   12 &   11 & \textbf{70} \\
\hline
\textbf{59}  & \textbf{13}  & \textbf{38}  & \textbf{68}  & \textbf{30}
\end{tabular}
\hspace{5mm}
\begin{tabular}{ccccc|c}
   5 &    7 &    6 &   25 &    0 & \textbf{43} \\
   2 &    1 &    1 &   15 &    4 & \textbf{23} \\
  29 &    0 &    4 &    2 &    0 & \textbf{35} \\
  37 &    0 &   20 &    0 &    0 & \textbf{57} \\
   0 &    0 &    8 &    2 &    0 & \textbf{10} \\
\hline
\textbf{73}  & \textbf{8}  & \textbf{39}  & \textbf{44}  & \textbf{4}
\end{tabular}
\hspace{5mm}
\begin{tabular}{ccccc}
\textbf{46} & \textbf{50} & \textbf{24} & \textbf{42} & \textbf{21}\\
\textbf{53} & \textbf{9} & \textbf{11} & \textbf{27} & \textbf{33}\\
\textbf{48} & \textbf{26} & \textbf{25} & \textbf{58} & \textbf{16}\\
\textbf{62} & \textbf{64} & \textbf{66} & \textbf{14} & \textbf{5}\\
\textbf{56} & \textbf{22} & \textbf{65} & \textbf{63} & \textbf{19}\\
~
\end{tabular}

\normalsize

\medskip
\noindent
SBHC$(4,1^4)$ \hspace{3mm} (underlying graph $K_3 \times K_4$)

\medskip
\noindent
\small
\begin{tabular}{cccc|c}
. & . & . & . & \\
. & . &   14 &    0 & \textbf{14} \\
. &    1 & . &    2 & \textbf{3} \\
. &   18 &    9 & . & \textbf{27} \\
\hline
 & \textbf{19}  & \textbf{23}  & \textbf{2}
\end{tabular}
\hspace{5mm}
\begin{tabular}{cccc|c}
. & . &   26 &    8 & \textbf{34} \\
. & . & . & . & \\
  22 & . & . &    9 & \textbf{31} \\
   2 & . &    6 & . & \textbf{8} \\
\hline
\textbf{24}  &  & \textbf{32}  & \textbf{17}
\end{tabular}
\hspace{5mm}
\begin{tabular}{cccc|c}
. &    4 & . &   25 & \textbf{29} \\
   1 & . & . &    0 & \textbf{1} \\
. & . & . & . & \\
   5 &    0 & . & . & \textbf{5} \\
\hline
\textbf{6}  & \textbf{4}  &  & \textbf{25}
\end{tabular}
\hspace{5mm}
\begin{tabular}{cccc|c}
. &   16 &    0 & . & \textbf{16} \\
   9 & . &   21 & . & \textbf{30} \\
   0 &   12 & . & . & \textbf{12} \\
. & . & . & . & \\
\hline
\textbf{9}  & \textbf{28}  & \textbf{21}  &
\end{tabular}
\hspace{5mm}
\begin{tabular}{cccc}
 & \textbf{20} & \textbf{26} & \textbf{33}\\
\textbf{10} &  & \textbf{35} & \textbf{0}\\
\textbf{22} & \textbf{13} &  & \textbf{11}\\
\textbf{7} & \textbf{18} & \textbf{15} & \\
~
\end{tabular}

\normalsize
\medskip
\noindent
SBHC$(5,1^5)$ \hspace{3mm} (underlying graph $K_3 \times K_5$)
\medskip

\noindent
\small
\tabcolsep=2.5pt
\begin{tabular}{ccccc|c}
. & . & . & . & . & \\
. & . &    0 &   45 &    3 & \textbf{48} \\
. &    0 & . &   10 &    1 & \textbf{11} \\
. &    3 &    0 & . &    3 & \textbf{6} \\
. &   24 &    0 &    4 & . & \textbf{28} \\
\hline
 & \textbf{27}  & \textbf{0}  & \textbf{59}  & \textbf{7}
\end{tabular}
\hspace{6mm}
\tabcolsep=2pt
\begin{tabular}{ccccc|c}
. & . &   49 &    9 &    0 & \textbf{58} \\
. & . & . & . & . & \\
   1 & . & . &    9 &    2 & \textbf{12} \\
   6 & . &    2 & . &    2 & \textbf{10} \\
  19 & . &    0 &   25 & . & \textbf{44} \\
\hline
\textbf{26}  &  & \textbf{51}  & \textbf{43}  & \textbf{4}
\end{tabular}
\hspace{6mm}
\begin{tabular}{ccccc|c}
. &    8 & . &   27 &    1 & \textbf{36} \\
   3 & . & . &   10 &   27 & \textbf{40} \\
. & . & . & . & . & \\
   3 &   12 & . & . &   19 & \textbf{34} \\
  12 &   10 & . &    0 & . & \textbf{22} \\
\hline
\textbf{18}  & \textbf{30}  &  & \textbf{37}  & \textbf{47}
\end{tabular}

\medskip
\noindent
\begin{tabular}{ccccc|c}
. &    9 &    1 & . &   15 & \textbf{25} \\
  42 & . &    2 & . &    2 & \textbf{46} \\
   4 &    4 & . & . &    0 & \textbf{8} \\
. & . & . & . & . & \\
   7 &    1 &   49 & . & . & \textbf{57} \\
\hline
\textbf{53}  & \textbf{14}  & \textbf{52}  &  & \textbf{17}
\end{tabular}
\hspace{5mm}
\begin{tabular}{ccccc|c}
. &   22 &    0 &   19 & . & \textbf{41} \\
   0 & . &    0 &    1 & . & \textbf{1} \\
   0 &    9 & . &    0 & . & \textbf{9} \\
  33 &    0 &   21 & . & . & \textbf{54} \\
. & . & . & . & . & \\
\hline
\textbf{33}  & \textbf{31}  & \textbf{21}  & \textbf{20}  &
\end{tabular}
\hspace{5mm}
\begin{tabular}{ccccc}
 & \textbf{39} & \textbf{50} & \textbf{55} & \textbf{16}\\
\textbf{45} &  & \textbf{2} & \textbf{56} & \textbf{32}\\
\textbf{5} & \textbf{13} &  & \textbf{19} & \textbf{3}\\
\textbf{42} & \textbf{15} & \textbf{23} &  & \textbf{24}\\
\textbf{38} & \textbf{35} & \textbf{49} & \textbf{29} & \\
~
\end{tabular}

\normalsize
\medskip
\noindent
SBHC$(6,1^6)$ \hspace{3mm} (underlying graph $K_3 \times K_6$)

\medskip
\noindent
\small
\tabcolsep=2pt
\begin{tabular}{cccccc|c}
. & . & . & . & . & . & \\
. & . &    6 &   12 &    1 &    9 & \textbf{28} \\
. &    5 & . &    4 &   12 &   16 & \textbf{37} \\
. &    2 &    3 & . &    0 &    2 & \textbf{7} \\
. &    3 &   11 &    1 & . &   36 & \textbf{51} \\
. &   19 &    2 &    4 &    2 & . & \textbf{27} \\
\hline
 & \textbf{29}  & \textbf{22}  & \textbf{21}  & \textbf{15}  & \textbf{63}
\end{tabular}
\hspace{5mm}
\begin{tabular}{cccccc|c}
. & . &    1 &    0 &    1 &    1 & \textbf{3} \\
. & . & . & . & . & . & \\
   5 & . & . &    0 &    0 &    0 & \textbf{5} \\
   8 & . &    1 & . &   14 &   16 & \textbf{39} \\
  70 & . &    0 &    0 & . &   14 & \textbf{84} \\
   5 & . &   43 &    0 &   32 & . & \textbf{80} \\
\hline
\textbf{88}  &  & \textbf{45}  & \textbf{0}  & \textbf{47}  & \textbf{31}
\end{tabular}
\hspace{5mm}
\begin{tabular}{cccccc|c}
. &    7 & . &    1 &   24 &   16 & \textbf{48} \\
  29 & . & . &    7 &    5 &    1 & \textbf{42} \\
. & . & . & . & . & . & \\
   3 &   52 & . & . &   10 &    1 & \textbf{66} \\
   0 &   11 & . &    0 & . &    1 & \textbf{12} \\
   3 &    0 & . &    9 &    4 & . & \textbf{16} \\
\hline
\textbf{35}  & \textbf{70}  &  & \textbf{17}  & \textbf{43}  & \textbf{19}
\end{tabular}

\medskip
\noindent
\begin{tabular}{cccccc|c}
. &    3 &    0 & . &   13 &   16 & \textbf{32} \\
   2 & . &    0 & . &   43 &   36 & \textbf{81} \\
  40 &   10 & . & . &    6 &   11 & \textbf{67} \\
. & . & . & . & . & . & \\
   7 &   18 &   42 & . & . &   20 & \textbf{87} \\
   8 &   21 &   20 & . &   12 & . & \textbf{61} \\
\hline
\textbf{57}  & \textbf{52}  & \textbf{62}  &  & \textbf{74}  & \textbf{83}
\end{tabular}
\hspace{5mm}
\begin{tabular}{cccccc|c}
. &   21 &   56 &    0 & . &    1 & \textbf{78} \\
   0 & . &    0 &    9 & . &    0 & \textbf{9} \\
   0 &   12 & . &    5 & . &    9 & \textbf{26} \\
   2 &    4 &    0 & . & . &   66 & \textbf{72} \\
. & . & . & . & . & . & \\
  38 &    1 &    0 &   10 & . & . & \textbf{49} \\
\hline
\textbf{40}  & \textbf{38}  & \textbf{56}  & \textbf{24}  &  & \textbf{76}
\end{tabular}
\hspace{5mm}
\begin{tabular}{cccccc|c}
. &   38 &    7 &    0 &   37 & . & \textbf{82} \\
  48 & . &    5 &    2 &    4 & . & \textbf{59} \\
  15 &   50 & . &    1 &    2 & . & \textbf{68} \\
   1 &    0 &    4 & . &    1 & . & \textbf{6} \\
   9 &    1 &    2 &    1 & . & . & \textbf{13} \\
. & . & . & . & . & . & \\
\hline
\textbf{73}  & \textbf{89}  & \textbf{18}  & \textbf{4}  & \textbf{44}  &
\end{tabular}

\medskip
\noindent
\begin{tabular}{cccccc}
 & \textbf{69} & \textbf{64} & \textbf{1} & \textbf{75} & \textbf{34}\\
\textbf{79} &  & \textbf{11} & \textbf{30} & \textbf{53} & \textbf{46}\\
\textbf{60} & \textbf{77} &  & \textbf{10} & \textbf{20} & \textbf{36}\\
\textbf{14} & \textbf{58} & \textbf{8} &  & \textbf{25} & \textbf{85}\\
\textbf{86} & \textbf{33} & \textbf{55} & \textbf{2} &  & \textbf{71}\\
\textbf{54} & \textbf{41} & \textbf{65} & \textbf{23} & \textbf{50} & \\
~
\end{tabular}

\normalsize

\normalsize
\medskip
\noindent
SBHC$(3,1^2)$ \hspace{3mm} (underlying graph $J_3$)

\medskip
\noindent
\begin{tabular}{ccc|c}
   3 &    6 &   5 & \textbf{14} \\
   2 &    . &    1 & \textbf{3} \\
   2 &   13 &    . & \textbf{15} \\
\hline
\textbf{7}  & \textbf{19}  & \textbf{6}
\end{tabular}
\hspace{5mm}
\begin{tabular}{ccc|c}
  16 &   . &   0 & \textbf{16} \\
   . &    . &    . &  \\
   2 &    . &    . & \textbf{2} \\
\hline
\textbf{18}  &   & \textbf{0}
\end{tabular}
\hspace{5mm}
\begin{tabular}{ccc|c}
   1 &   11 &    . & \textbf{12} \\
   8 &   . &    . & \textbf{8} \\
   . &    . &    . & \\
\hline
\textbf{9}  & \textbf{11}  & 
\end{tabular}
\hspace{5mm}
\begin{tabular}{ccc}
\textbf{20} & \textbf{17} & \textbf{5}\\
\textbf{10} &  & \textbf{1}\\
\textbf{4} & \textbf{13} & \\
~
\end{tabular}

\normalsize
\medskip
\noindent
SBHC$(4,1^3)$ \hspace{3mm} 
(underlying graph $J_4$)

\medskip
\noindent
\small
\begin{tabular}{cccc|c}
   1 &   13 &    4 &    2 & \textbf{20} \\
   1 &    . &   19 &   10 & \textbf{30} \\
  25 &    7 &    . &    0 & \textbf{32} \\
   6 &    1 &    1 &    . & \textbf{8} \\
\hline
\textbf{33}  & \textbf{21}  & \textbf{24}  & \textbf{12}
\end{tabular}
\hspace{4mm}
\begin{tabular}{cccc|c}
   3 &    . &    0 &    1 & \textbf{4} \\
   . &    . &    . &    . &  \\
   3 &    . &    . &   15 & \textbf{18} \\
  25 &    . &    0 &    . & \textbf{25} \\
\hline
\textbf{31}  &   & \textbf{0}  & \textbf{16}
\end{tabular}
\hspace{4mm}
\begin{tabular}{cccc|c}
   1 &    4 &    . &   31 & \textbf{36} \\
  22 &    . &    . &    4 & \textbf{26} \\
   . &    . &    . &    . &  \\
   6 &    1 &    . &    . & \textbf{7} \\
\hline
\textbf{29}  & \textbf{5}  &   & \textbf{35}
\end{tabular}
\hspace{4mm}
\begin{tabular}{cccc|c}
   1 &    0 &   18 &    . & \textbf{19} \\
   0 &    . &    9 &    . & \textbf{9} \\
  10 &    3 &    . &    . & \textbf{13} \\
   . &    . &    . &    . &  \\
\hline
\textbf{11}  & \textbf{3}  & \textbf{27} 
\end{tabular}
\hspace{4mm}
\begin{tabular}{cccc}
\textbf{6} & \textbf{17} & \textbf{22} & \textbf{34}\\
\textbf{23} &  & \textbf{28} & \textbf{14}\\
\textbf{38} & \textbf{10} &  & \textbf{15}\\
\textbf{37} & \textbf{2} & \textbf{1} & \\
~
\end{tabular}



\begin{thebibliography}{99}


\bibitem{handbook}
C.J.~Colbourn and J.H.~Dinitz, {\em Handbook of Combinatorial Designs}, 2nd ed., CRC Press, 2007.

\bibitem{DSG}
P.J.~Dukes and J.A.~Short-Gershman, A complete existence theory for Sarvate-Beam triple systems. \emph{Australas. J. Combin.} 54 (2012), 261--272.

\bibitem{pbd456}
H.~Lenz, Some remarks on pairwise balanced designs. {\em Mitt. Math. Sem. Giessen} 165 (1984), 49--62.

\bibitem{SB1}
D.~Sarvate and W.~Beam, A new type of block design. {\em Bull. Inst. Combin. Appl.} 50 (2007), 26--28.

\bibitem{St2}
R.G.~Stanton, Sarvate-Beam Triple Systems for $v \equiv 2 \pmod{3}$. {\em J. Combin. Math. Combin. Comput.} 61 (2007), 129--134.

\bibitem{Wilson}
R.M.~Wilson, An existence theory for pairwise balanced designs: II, The structure of PBD-closed sets and the existence conjectures. {\em J. Combin. Theory Ser. A} 13 (1972), 246--273.

\end{thebibliography}
\end{document}